\documentclass[a4paper,10pt, reqno]{amsart}
\usepackage[backref]{hyperref}
\usepackage{a4wide}
\usepackage{amsthm}
 \newtheorem{theorem}{Theorem}[section]
 
 \newtheorem{lemma}[theorem]{Lemma}
 \newtheorem{corollary}[theorem]{Corollary}

\begin{document}

\title{A note on Diophantine approximation with Gaussian primes}
\author{Stephan Baier}
\address{Stephan Baier, Jawaharlal Nehru University, School of Physical Sciences,
Delhi 11067, India}

\email{sbaier@math.tifr.res.in}


\maketitle

{\bf Abstract:} We investigate the distribution of $p\theta$ modulo 1, where
$\theta$ is a complex number which is not contained in $\mathbb{Q}(i)$, and $p$ runs over the Gaussian primes. 

\section{Introduction}
Let $\theta\in \mathbb{R}$ be an irrational number. Then the continued fraction expansion of $\theta$ yields infinitely many natural
numbers $q$ such that
$$
\left|\theta-\frac{a}{q}\right| \le q^{-2},
$$
where $(a,q)=1$. In other words, for infinitely many $q\in \mathbb{N}$, we have
$$
||q \theta||\le q^{-1},
$$
where $||x||$ is the distance of $x$ to the nearest integer. The problem of approximating irrational numbers by rational numbers with prime 
denominator is considerable more difficult und has a long history. The question is for which $\gamma>0$ one can prove the infinitude
of primes $p$ such that
\begin{equation} \label{approxi}
||p \theta||\le p^{-\gamma+\varepsilon}.
\end{equation}
The first results in this direction are due to Vinogradov \cite{Vin} who showed that $\gamma=1/5$ is admissable. Vaughan \cite{Vau} improved this exponent to 
$\gamma=1/4$ using his famous identity for the von Mangoldt function. 
It should be noted that using Vaughan's method, an {\it asymptotic} result of the following form can be established. 

\begin{theorem} \label{Vaughan} Let $\theta\in \mathbb{R}$ be irrational and $\varepsilon>0$ be an arbitrary constant. 
There exists an infinite 
increasing sequence of natural numbers $(N_k)_{k\in \mathbb{N}}$ such that
\begin{equation} \label{asy}
\sum\limits_{\substack{N_k/2<p\le N_k\\ ||p\theta||\le \delta_k}} 1 \sim 2\delta_k \sum\limits_{N_k/2<p\le N_k} 1 
\quad \mbox{as } k\rightarrow \infty
\end{equation}
if 
\begin{equation} \label{range}
N_k^{-1/4+\varepsilon}\le \delta_k\le 1/2.
\end{equation}
\end{theorem}

The next important step was Harman's work \cite{Har} in which he used his sieve method to show
that \eqref{approxi} holds for infinitely many primes $p$ if $\gamma=3/10$. 
Harman's method doesn't imply the asymptotic \eqref{asy} for $\delta_k=N_k^{-3/10+\varepsilon}$ since it uses a lower
bound sieve. However, Harman's sieve can be employed to recover Vaughan's result and hence \eqref{asy} for the same
$\delta_k$-range as in \eqref{range}. We further mention the work of Heath-Brown and Jia \cite{HeJ} who
used bounds for Kloosterman sums to obtain a further improvement of the exponent to $\gamma=16/49=1/3-0.0068...$. Finally,
the exponent $\gamma=1/3$ was achieved in a landmark paper by Matom\"aki \cite{Mat} who incorporated the Kuznetsov formula into the method to bound
sums of Kloosterman sums. This exponent $\gamma=1/3$ is considered to be the limit of currently available techniques.

The purpose of this note is to formulate the problem for Gaussian primes and establish a result corresponding to 
Theorem \ref{Vaughan} in this context, thereby proving the infinitude of Gaussian primes satisfying an inequality 
corresponding to \eqref{approxi}. 
To this end, we shall apply a version of Harman's sieve for 
$\mathbb{Z}[i]$. Our method will require additional counting arguments, as compared to the classical method. 
Our results are stated in section \ref{con}.
 
As usual, throughout this paper, $\varepsilon$ denotes an arbitrarily small positive number which may change from line 
to line.

\section{Setup}
Throughout the sequel, $\theta$ is a complex number such that $\theta\not \in \mathbb{Q}(i)$, and $x\ge 1$. We compare the quantities
$$
S(x,\delta):=\sum\limits_{\substack{x/2<\mathcal{N}(p)\le x\\ ||p\theta||\le \delta}} 1 
$$
and
$$
S(x):=\sum\limits_{\substack{x/2<\mathcal{N}(p)\le x}} 1.
$$
Here the sums run over Gaussian primes $p$, $\mathcal{N}(n)$ denotes the norm of $n\in \mathbb{Z}[i]$, $0<\delta\le 1/2$, and we define
$$
||z||:=\max\left\{||\Re(z)||,||\Im(z)||\right\},
$$
where $\Re(z)$ is the real part and $\Im(z)$ is the imaginary part of $z\in \mathbb{C}$. Hence, $||z||$ measures the distance of $z$ to the 
nearest Gaussian integer with respect to the supremum norm. 
By the prime number theorem for Gaussian primes, 
\begin{equation} \label{primenumber}
S(x) \sim 2 \cdot \frac{x}{\log x} \quad \mbox{as } x\rightarrow \infty.
\end{equation}

Our goal is to construct an infinite increasing sequence $\left(N_k\right)_{k\in \mathbb{N}}$ of natural numbers such that
$$
S\left(N_k,\delta_k\right) \sim 4\delta_k^2 S(N_k) \quad \mbox{as } k \rightarrow \infty
$$
for $N_k^{-\gamma+\varepsilon}\le \delta_k\le 1/2$, where $\gamma$ is a suitable positive number.  
In this paper, we shall show that $\gamma=1/24$ is admissable.
The full analog of Theorem \eqref{Vaughan} would be achieved if $\gamma=1/24$ could be replaced by $\gamma=1/8$ 
(the exponent 1/4 in \eqref{range} needs to be  halfed because our setting is 2-dimensional). 

\section{Application of Harman's sieve for $\mathbb{Z}[i]$}
In the following, let $A$ be a finite set of non-zero Gaussian integers, $P$ be a subset of the set $\mathbb{P}$ of Gaussian primes and $z$ be a positive parameter. By 
$\mathcal{S}(A,P,z)$ we denote the number
of elements of $A$ which are coprime to $P(z)$, the product of all Gaussian primes in $P$ with norm $\le z$, i.e.
$$
\mathcal{S}(A,P,z)=\sharp\{n\in A\ :\ p\not| \; n \mbox{ for all } p\in P \mbox{ such that } \mathcal{N}(p)\le z\}.
$$
The following is a version of Harman's sieve for $\mathbb{Z}[i]$.

\begin{theorem}[Harman] \label{Harsie} Let $A,B$ be finite sets of non-zero Gaussian integers with norm $\le x$. Suppose for any sequences $(a_n)_{n\in \mathbb{Z}[i]}$ and 
$(b_n)_{n\in \mathbb{Z}[i]}$ of complex numbers satisfying $|a_n|,|b_n|\le 1$ 
the following hold:
\begin{equation} \label{t1}
\sum\limits_{\substack{\mathcal{N}(m)\le M\\ mn\in A}} a_m = \lambda \sum\limits_{\substack{\mathcal{N}(m)\le M\\ mn\in B}} a_m + O\left(Y\right),
\end{equation}
\begin{equation} \label{t2}
\sum\limits_{\substack{x^{\alpha}<\mathcal{N}(m)\le x^{\alpha+\beta}\\ mn\in A}} a_mb_n = 
\lambda \sum\limits_{\substack{x^{\alpha}<\mathcal{N}(m)\le x^{\alpha+\beta}\\ mn\in B}} a_mb_n + O\left(Y\right)
\end{equation}
for some $\lambda,Y>0$, $\alpha>0$, $0<\beta\le 1/2$ and $M>x^{\alpha}$. Then we have 
\begin{equation} \label{super}
\mathcal{S}(A,\mathbb{P},x^{\beta})=\lambda \mathcal{S}(B,\mathbb{P},x^{\beta})+O\left(Y\log^3 x\right).
\end{equation}
\end{theorem}

\begin{proof}
The proof is parallel to that of Harman's sieve for the classical case, Theorem 3.1. in \cite{HPD} 
(Fundamental Theorem) with $R=1$ and $c_r=1$, making repeated use of the Buchstab identity in the setting of $\mathbb{Z}[i]$ (see Chapter 11
in \cite{HPD}). Therefore, we omit the details.
\end{proof}

In the usual terminology, the sums in \eqref{t1} are referred to as type I bilinear sums, and the sums in \eqref{t2} as type II bilinear sums.

We shall apply Theorem \ref{Harsie} to the situation when 
$$
A:=\left\{n\in \mathbb{Z}[i]\ :\ x/2<\mathcal{N}(n)\le x,\ ||n\theta || \le \delta\right\} \quad
\mbox{with } 0<\delta\le 1/2,
$$
$$
B:=\left\{n\in \mathbb{Z}[i]\ :\ x/2<\mathcal{N}(n)\le x\right\} \quad \mbox{and} \quad \beta=\frac{1}{2}.
$$
The parameters $\alpha$ and $M$ will later be chosen suitably. We note that 
\begin{equation} \label{Sxd}
S(x,\delta)=\sum\limits_{\substack{x/2<\mathcal{N}(p)\le x\\ ||n\theta ||\le \delta}} 1 = \mathcal{S}(A,\mathbb{P},x^{1/2})  
\end{equation}
and
\begin{equation} \label{Sx}
S(x)=\sum\limits_{\substack{x/2<\mathcal{N}(p)\le x}} 1 = \mathcal{S}(B,\mathbb{P},x^{1/2}).
\end{equation}

\section{Detecting small $||mn\theta||$}
We observe that
$$
||mn\theta||\le \delta \Longleftrightarrow \left([\delta-\Re(mn\theta)]-[-\delta-\Re(mn\theta)]\right)\left([\delta-\Im(mn\theta)]-
[-\delta-\Im(mn\theta)]\right)=1.
$$
Hence, the type I sum in question can be written in the form
\begin{equation*}
\begin{split}
\sum\limits_{\substack{\mathcal{N}(m)\le M\\ mn\in A}} a_m = & \sum\limits_{\mathcal{N}(m)\le M} a_m \cdot
\sum\limits_{x/(2\mathcal{N}(m))<N(n)\le x/\mathcal{N}(m)} 
\left([\delta-\Re(mn\theta)]-[-\delta-\Re(mn\theta)]\right)\times\\ & \left([\delta-\Im(mn\theta)]-[-\delta-\Im(mn\theta)]\right).
\end{split}
\end{equation*}
Further, using $[x]=x-\psi(x)-1/2$, the inner sum over $n$ can be expressed in the form
\begin{equation*} 
\begin{split}
& \sum\limits_{x/(2\mathcal{N}(m))<N(n)\le x/\mathcal{N}(m)}  
\left([\delta-\Re(mn\theta)]-[-\delta-\Re(mn\theta)]\right)\times\\ & \left([\delta-\Im(nn\theta)]-[-\delta-\Im(mn\theta)]\right)\\
= & 4\delta^2 \sum\limits_{x/(2\mathcal{N}(m))<N(n)\le x/\mathcal{N}(m)} 1-\\ & 2\delta 
\sum\limits_{x/(2\mathcal{N}(m))<N(n)\le x/\mathcal{N}(m)} \left(\psi\left(\delta-\Im(mn\theta)\right)-\psi\left(-\delta-\Im(mn\theta)\right)\right) 
- \\ & 2\delta 
\sum\limits_{x/(2\mathcal{N}(m))<N(n)\le x/\mathcal{N}(m)} \left(\psi\left(\delta-\Re(mn\theta)\right)-\psi\left(-\delta-\Re(mn\theta)\right)\right)+\\
& \sum\limits_{x/(2\mathcal{N}(m))<N(n)\le x/\mathcal{N}(m)} \left(\psi\left(\delta-\Re(mn\theta)\right)-\psi\left(-\delta-\Re(mn\theta)\right)\right)\times\\ &   
\left(\psi\left(\delta-\Im(mn\theta)\right)-\psi\left(-\delta-\Im(mn\theta)\right)\right)\\
= &  4\delta^2 \sum\limits_{x/(2\mathcal{N}(m))<N(n)\le x/\mathcal{N}(m)} 1 - 2\delta S_1- 2\delta S_2+ S_3, 
\end{split}
\end{equation*}
say. Next, we approximate the function $\psi(x)$ by a trigonomtrical polynomial using the following lemma due to Vaaler (see \cite{GKo}, Theorem A6). 

\begin{lemma}[Vaaler] \label{Vaaler} For $0<|t|<1$ let
$$
W(t)=\pi t(1-|t|) \cot \pi t +|t|.
$$
Fix a natural number $J$. For $x\in \mathbb{R}$ define 
$$
\psi^{\ast}(x):=-\sum\limits_{1\le |j|\le J} (2\pi i j)^{-1}W\left(\frac{j}{J+1}\right)e(jx)
$$
and
$$
\sigma(x):=\frac{1}{2J+2} \sum\limits_{|j|\le J} \left(1-\frac{|j|}{J+1}\right)e(jx).
$$
Then $\sigma(x)$ is non-negative, and we have 
$$
|\psi^{\ast}(x)-\psi(x)|\le \sigma(x)
$$
for all real numbers $x$. 
\end{lemma}

Throughout the sequel, $J$ denotes a natural number such that $J\ge \delta^{-1}$ which will be fixed in section \ref{finest}. From 
Lemma \ref{Vaaler}, we deduce that
\begin{equation*}
\begin{split}
S_1\ll & \frac{x/\mathcal{N}(m)}{J}+ \sum\limits_{1\le |j|\le J} \frac{1}{|j|} \cdot  \Big|
\sum\limits_{x/(2\mathcal{N}(m))<N(n)\le x/\mathcal{N}(m)} \left(e\left(j(\delta-\Im(mn\theta))\right)-e\left(j(-\delta-\Im(mn\theta)\right)
\right) \Big| \\
\ll & \frac{x/\mathcal{N}(m)}{J}+ \sum\limits_{1\le |j|\le J} \frac{1}{|j|}\cdot \Big|
\sum\limits_{x/(2\mathcal{N}(m))<N(n)\le x/\mathcal{N}(m)} \left(e\left(j\delta\right)-e\left(-j\delta\right)\right) \cdot 
e\left(-j\Im(mn\theta)\right)\Big|\\ 
\ll & \frac{x/\mathcal{N}(m)}{J}+ \sum\limits_{1\le |j|\le J} \min\{\delta, |j|^{-1}\} \cdot \Big|
\sum\limits_{x/(2\mathcal{N}(m))<N(n)\le x/\mathcal{N}(m)} e\left(j\Im(mn\theta)\right)\Big|.
\end{split}
\end{equation*}
In a similar way, we obtain
\begin{equation*}
\begin{split}
S_2\ll \frac{x/\mathcal{N}(m)}{J}+ \sum\limits_{1\le |j|\le J} \min\{\delta, |j|^{-1}\} \cdot \Big|
\sum\limits_{x/(2\mathcal{N}(m))<N(n)\le x/\mathcal{N}(m)} e\left(j\Re(mn\theta)\right)\Big|
\end{split}
\end{equation*}
and
\begin{equation*}
\begin{split}
S_3\ll & \frac{x/\mathcal{N}(m)}{J^2} + \frac{1}{J}\cdot \sum\limits_{1\le |j_1|\le J} \min\{\delta,|j_1|^{-1}\} \cdot \Big|
\sum\limits_{x/(2\mathcal{N}(m))<N(n)\le x/\mathcal{N}(m)} e\left(j_1\Im(mn\theta)\right)\Big| + \\
& \frac{1}{J}\cdot \sum\limits_{1\le |j_2|\le J} \min\{\delta, |j_2|^{-1}\} \cdot \Big|
\sum\limits_{x/(2\mathcal{N}(m))<N(n)\le x/\mathcal{N}(m)} e\left(j_2\Re(mn\theta)\right)\Big|+\\
& \sum\limits_{\substack{1\le |j_1|\le J\\ 1\le |j_2|\le J}} \min\{\delta, |j_1|^{-1}\} \cdot \min\{\delta, |j_2|^{-1}\} \cdot \Big|
\sum\limits_{x/(2\mathcal{N}(m))<N(n)\le x/\mathcal{N}(m)} e\left(j_1\Im(mn\theta)+j_2\Re(mn\theta)\right)\Big|.
\end{split}
\end{equation*}
Summing over $m$ and using $|a_m|\le 1$ and $J\ge \delta^{-1}$, we get
\begin{equation} \label{typeI}
\begin{split}
\sum\limits_{\substack{\mathcal{N}(m)\le M\\ mn\in A}} a_m = & 4\delta^2 \sum\limits_{\mathcal{N}(m)\le M} a_m
\sum\limits_{x/(2\mathcal{N}(m))<N(n)\le x/\mathcal{N}(m)} 1\\ 
& + O\left(\delta x^{1+\varepsilon}J^{-1}+\delta (E_1+E_2)+E_3\right)\\
= &
4\delta^2 \sum\limits_{\substack{\mathcal{N}(m)\le M\\ mn\in B}} a_m+
O\left(\delta x^{1+\varepsilon}J^{-1}+\delta (E_1+E_2)+E_3\right), 
\end{split}
\end{equation}
where
\begin{equation*}
E_1=\sum\limits_{1\le |j|\le J} \min\{\delta,|j|^{-1}\} \cdot \sum\limits_{\mathcal{N}(m)\le M} \Big|
\sum\limits_{x/(2\mathcal{N}(m))<N(n)\le x/\mathcal{N}(m)} e\left(j\Im(mn\theta)\right)\Big|,
\end{equation*}
\begin{equation*}
E_2=\sum\limits_{1\le |j|\le J} \min\{\delta,|j|^{-1}\} \cdot \sum\limits_{\mathcal{N}(m)\le M} \Big|
\sum\limits_{x/(2\mathcal{N}(m))<N(n)\le x/\mathcal{N}(m)} e\left(j\Re(mn\theta)\right)\Big|
\end{equation*}
and 
\begin{equation}
\begin{split}
E_3= & \sum\limits_{1\le |j_1|\le J} \sum\limits_{1\le |j_2|\le J} \min\{\delta,|j_1|^{-1}\} \cdot \min\{\delta,|j_2|^{-1}\} 
\times\\ & 
\sum\limits_{\mathcal{N}(m)\le M} \Big|
\sum\limits_{x/(2\mathcal{N}(m))<N(n)\le x/\mathcal{N}(m)}e\left(j_1\Im(mn\theta)+j_2\Re(mn\theta)\right)\Big|.
\end{split}
\end{equation}

In a similar way, using $|a_m|,|b_n|\le 1$ and $J\ge \delta^{-1}$, we derive the asymptotic estimate
\begin{equation} \label{typeII}
\begin{split}
\sum\limits_{\substack{x^{\alpha}<\mathcal{N}(m)\le x^{\alpha+\beta}\\ mn\in A}} a_mb_n = &
4\delta^2 \sum\limits_{\substack{x^{\alpha}<\mathcal{N}(m)\le x^{\alpha+\beta}\\ mn\in B}} a_mb_n+ \\ &
O\left(\delta x^{1+\varepsilon}J^{-1}+\delta(F_1+G_1+F_2+G_2)+F_3+G_3\right), 
\end{split}
\end{equation}
where
\begin{equation*}
F_1= \sum\limits_{1\le |j|\le J} \min\{\delta,|j|^{-1}\} \cdot 
\Big|\sum\limits_{\substack{x^{\alpha}<\mathcal{N}(m)\le x^{\alpha+\beta}\\ mn\in B}} 
a_mb_ne\left(j\Im(mn\theta)\right)\Big|,
\end{equation*}
\begin{equation*}
F_2= \sum\limits_{1\le |j|\le J} \min\{\delta,|j|^{-1}\} \cdot 
\Big|\sum\limits_{\substack{x^{\alpha}<\mathcal{N}(m)\le x^{\alpha+\beta}\\ mn\in B}} 
a_mb_ne\left(j\Re(mn\theta)\right)\Big|,
\end{equation*}
\begin{equation}
\begin{split}
F_3= & \sum\limits_{1\le |j_1|\le J} \sum\limits_{1\le |j_2|\le J} \min\{\delta,|j_1|^{-1}\} \cdot \min\{\delta,|j_2|^{-1}\}\times\\
& 
\Big|\sum\limits_{\substack{x^{\alpha}<\mathcal{N}(m)\le x^{\alpha+\beta}\\ mn\in B}} a_mb_ne\left(j_1\Im(mn\theta)+j_2\Re(mn\theta)\right)\Big|,
\end{split}
\end{equation}
and $G_i$ $(i=1,2,3)$ are the same sums as $F_i$ with $a_mb_n$ replaced by $1$. The sums $G_i$ 
$(i=1,2,3)$ can be treated similarly as the sums $F_i$. Therefore, in the following, we focus only on the treatments of $E_i$ and $F_i$ 
$(i=1,2,3)$. 

\section{Transformations of the sums $E_i$ and $F_i$} \label{trans}

We note that 
\begin{equation} \label{E1E2}
E_1=E_2.
\end{equation} 
We further have, by breaking the $|j|$-range into $O(\log 2J)$ dyadic intervals, 
\begin{equation} \label{E1H}
E_1\ll (\log 2J) \cdot \sup\limits_{1\le H\le J} \min\{\delta,H^{-1}\} \cdot E_1(H),
\end{equation}
where
\begin{equation}
E_1(H)=\sum\limits_{1\le |j|\le H} \sum\limits_{\mathcal{N}(m)\le M} \Big|
\sum\limits_{x/(2\mathcal{N}(m))<N(n)\le x/\mathcal{N}(m)} e\left(j\Im(mn\theta)\right)\Big|.
\end{equation}
Similarly,
\begin{equation} \label{E3rel}
E_3\ll (\log 2J)^2\cdot \sup\limits_{\substack{1\le |H_1|\le J\\ 1\le |H_2|\le J}} \min\{\delta,H_1^{-1}\} \cdot
\min\{\delta,H_2^{-1}\} \cdot E_3(H_1,H_2),
\end{equation}
where 
\begin{equation}
E_3(H_1,H_2)=\mathop{\sum\limits_{|j_1|\le H_1} 
\sum\limits_{|j_2|\le H_2}}_{(j_1,j_2)\not= (0,0)} \sum\limits_{\mathcal{N}(m)\le M} \Big|
\sum\limits_{x/(2\mathcal{N}(m))<N(n)\le x/\mathcal{N}(m)}e\left(j_1\Im(mn\theta)+j_2\Re(mn\theta)\right)\Big|.
\end{equation}
We note that
\begin{equation} \label{Edef}
E_3(H_1,H_2)=\sum\limits_{\substack{j\not=0\\ |\Re(j)|\le H_1\\ |\Im(j)|\le H_2}} \sum\limits_{\mathcal{N}(m)\le M}
\Big| \sum\limits_{x/(2\mathcal{N}(m))<N(n)\le x/\mathcal{N}(m)} e\left(\Im(jmn\theta)\right)\Big|
\end{equation}
and hence, 
\begin{equation} \label{E1E3}
E_1(H)= E_3(H,1/2). 
\end{equation}
Thus, it suffices to estimate $E_3(H_1,H_2)$ for $H_1\ge 1$ and $H_2\ge 1/2$ to bound $E_1$, $E_2$ and $E_3$.

Similarly,
\begin{equation} \label{F1F2}
F_1=F_2
\end{equation}
and 
\begin{equation} \label{F1H}
F_1\ll (\log 2J) \cdot \sup\limits_{1\le H\le J} \min\{\delta,H^{-1}\} \cdot F_1(H),
\end{equation}
where
\begin{equation}
F_1(H)=\sum\limits_{1\le |j|\le H} \Big|\sum\limits_{\substack{x^{\alpha}<\mathcal{N}(m)\le x^{\alpha+\beta}\\ mn\in B}} 
a_mb_ne\left(j\Im(mn\theta)\right)\Big|,
\end{equation}
and 
\begin{equation} \label{F3rel}
F_3= (\log 2J)^2\cdot \sup\limits_{\substack{1\le |H_1|\le J\\ 1\le |H_2|\le J}} \min\{\delta,H_1^{-1}\} \cdot
\min\{\delta,H_2^{-1}\} \cdot F_3(H_1,H_2),
\end{equation}
where 
\begin{equation} \label{F3}
\begin{split}
F_3(H_1,H_2)= & \mathop{\sum\limits_{|j_1|\le H_1} \sum\limits_{|j_2|\le H_2}}_{(j_1,j_2)
\not=(0,0)}
\Big|\sum\limits_{\substack{x^{\alpha}<\mathcal{N}(m)\le x^{\alpha+\beta}\\ mn\in B}} 
a_mb_ne\left(j_1\Im(mn\theta)+j_2\Re(mn\theta)\right)\Big|\\
= & \sum\limits_{\substack{j\not=0\\ |\Re(j)|\le H_1\\ |\Im(j)|\le H_2}} \Big|\sum\limits_{\substack{x^{\alpha}<\mathcal{N}(m)\le x^{\alpha+\beta}\\ mn\in B}} 
a_mb_ne\left(\Im(jmn\theta)\right)\Big|
\end{split}
\end{equation}
and hence,
\begin{equation} \label{F1F3}
F_1(H)=F_3(H,1/2).
\end{equation}
Thus, it suffices to estimate $F_3(H_1,H_2)$ for $H_1\ge 1$ and $H_2\ge 1/2$ to bound $F_1$, $F_2$ and $F_3$.

So we have reduced the problem to bounding the type I sums $E_3(H_1,H_2)$ and the type II sums $F_3(H_1,H_2)$.

\section{Treatment of type II sums} \label{treat}
To treat the type II sums, we first reduce them to type I sums. We begin by splitting $F_3(H_1,H_2)$ into subsums of the form
\begin{equation} \label{subsums}
F_3(H_1,H_2,K,K'):=\sum\limits_{\substack{j\not=0\\ |\Re(j)|\le H_1\\ |\Im(j)|\le H_2}} \Big|\sum\limits_{\substack{K<\mathcal{N}(m)\le K'\\ mn\in B}} 
a_mb_ne\left(\Im(jmn\theta)\right)\Big|,
\end{equation}
where $K<K'\le 2K$. Next, we apply the Cauchy-Schwarz inequality, getting
$$
F_3(H_1,H_2,K,K')^2\ll H_1H_2K\cdot \sum\limits_{\substack{j\not=0\\ |\Re(j)|\le H_1\\ |\Im(j)|\le H_2}} 
\sum\limits_{\substack{K<\mathcal{N}(m)\le K'}} \Big| \sum\limits_{x/(2\mathcal{N}(m))<\mathcal{N}(n)\le x/\mathcal{N}(m)} 
b_ne\left(\Im(jmn\theta)\right)\Big|^2, 
$$
where we use the bound $|a_m|\le 1$. 
Expanding the square and re-arranging summation, we get
\begin{equation} \label{F}
\begin{split}
& F_3(H_1,H_2,K,K')^2\ll H_1H_2K\cdot  \sum\limits_{\substack{j\not=0\\ |\Re(j)|\le H_1\\ |\Im(j)|\le H_2}}\sum\limits_{x/(2K')<\mathcal{N}(n_1),\mathcal{N}(n_2) \le x/K} b_{n_1}\overline{b_{n_2}}\times\\ 
& \sum\limits_{\max\{K,x/(2\mathcal{N}(n_1)),x/(2\mathcal{N}(n_2))\}<\mathcal{N}(m)\le \min\{K',x/\mathcal{N}(n_1),x/\mathcal{N}(n_2)\}}
e\left(\Im(jm(n_1-n_2)\theta)\right)\\
\ll & H_1^2H_2^2Kx+ H_1H_2K\cdot \sum\limits_{\substack{j\not=0\\ |\Re(j)|\le H_1\\ |\Im(j)|\le H_2}} 
\sum\limits_{\substack{x/(2K')<\mathcal{N}(n_1),\mathcal{N}(n_2) \le x/K\\ n_1\not=n_2}} 
\\ 
& \Big| \sum\limits_{\max\{K,x/(2\mathcal{N}(n_1)),x/(2\mathcal{N}(n_2))\}<\mathcal{N}(m)\le \min\{K',x/\mathcal{N}(n_1),x/\mathcal{N}(n_2)\}}
e\left(\Im(jm(n_1-n_2)\theta)\right)\Big|\\
\ll & H_1^2H_2^2Kx+ H_1H_2K\cdot \sum\limits_{0<\mathcal{N}(n)\le 4(H_1^2+H_2^2)x/K} \sum\limits_{j|n} 
\sum\limits_{\substack{x/(2K')<\mathcal{N}(n_1),\mathcal{N}(n_2) \le x/K\\ n/j=n_1-n_2}} \\
& \Big| \sum\limits_{\max\{K,x/(2\mathcal{N}(n_1)),x/(2\mathcal{N}(n_2))\}<\mathcal{N}(m)\le \min\{K',x/\mathcal{N}(n_1),x/\mathcal{N}(n_2)\}}
e\left(\Im(mn\theta)\right)\Big|.\\
\end{split}
\end{equation}
Here the second line arrives by isolating the diagonal contribution of $n_1=n_2$ and using the bound $|b_n|\le 1$,
and the third line arrives by writing $n=j(n_1-n_2)$.

\section{Estimating sums of linear exponential sums}
Our next task is to bound linear exponential sums of the form 
\begin{equation} \label{1}
\begin{split}
& \sum\limits_{\tilde{y}<\mathcal{N}(m)\le y} e\left(\Im(m\kappa)\right) =  \sum\limits_{\substack{(m_1,m_2)\in \mathbb{Z}^2\\ 
\tilde{y}<m_1^2+m_2^2\le y}} 
e\left(\Re(\kappa)m_2+\Im(\kappa)m_1\right)\\ = & \sum\limits_{\substack{(m_1,m_2)\in \mathbb{Z}^2\\ m_1^2+m_2^2\le y}} 
e\left(\Re(\kappa)m_2+\Im(\kappa)m_1\right) - \sum\limits_{\substack{(m_1,m_2)\in \mathbb{Z}^2\\ m_1^2+m_2^2\le \tilde{y}}} 
e\left(\Re(\kappa)m_2+\Im(\kappa)m_1\right),
\end{split}
\end{equation}
where $\kappa$ is a complex number and $0\le \tilde{y}<y$. Here we use the following simple slicing argument. We have
\begin{equation} \label{m_1m_2}
\begin{split} 
& \sum\limits_{\substack{(m_1,m_2)\in \mathbb{Z}^2\\ 
m_1^2+m_2^2\le y}} 
e\left(\Re(\kappa)m_2+\Im(\kappa)m_1\right)\\ = & \sum\limits_{-\sqrt{y}\le m_1\le \sqrt{y}} \sum\limits_{-\sqrt{y-m_1^2}\le  m_2 \le \sqrt{y-m_1^2} }
e\left(\Re(\kappa)m_2+\Im(\kappa)m_1\right) \\
\ll & \sum\limits_{-\sqrt{y}\le m_1\le \sqrt{y}}  \Big|\sum\limits_{-\sqrt{y-m_1^2}\le  m_2 \le \sqrt{y-m_1^2}} 
e\left(\Re(\kappa)m_2\right) \Big|\\
\ll & y^{1/2} \cdot \min\left\{||\Re(\kappa)||^{-1},\sqrt{y}\right\},
\end{split}
\end{equation}
where we use the classical bound
$$
\sum\limits_{a<m\le b} e(mz) \ll \min\left\{b-a+1,||z||^{-1}\right\}
$$
for linear exponential sums. Similarly, by interchanging the rules of $m_1$ and $m_2$, we get
\begin{equation*}
\begin{split} 
 \sum\limits_{\substack{(m_1,m_2)\in \mathbb{Z}^2\\ 
m_1^2+m_2^2\le y}} 
e\left(\Re(\kappa)m_2+\Im(\kappa)m_1\right) \ll y^{1/2} \cdot \min\left\{||\Im(\kappa)||^{-1},\sqrt{y}\right\}.
\end{split}
\end{equation*}
Taking the geometric mean of these two estimates gives 
\begin{equation*}
\sum\limits_{\substack{(m_1,m_2)\in \mathbb{Z}^2\\ m_1^2+m_2^2\le y}} 
e\left(\Re(\kappa)m_2+\Im(\kappa)m_1\right) \ll y^{1/2} \cdot 
\min\left\{||\Im(\kappa)||^{-1},\sqrt{y}\right\}^{1/2} \cdot \min\left\{||\Re(\kappa)||^{-1},\sqrt{y}\right\}^{1/2}.
\end{equation*}
Using \eqref{1}, we deduce that
\begin{equation} \label{lin}
\sum\limits_{\tilde{y}<\mathcal{N}(m)\le y} e\left(\Im(m\kappa)\right) \ll y^{1/2} \cdot 
\min\left\{||\Im(\kappa)||^{-1},\sqrt{y}\right\}^{1/2} \cdot \min\left\{||\Re(\kappa)||^{-1},\sqrt{y}\right\}^{1/2}.
\end{equation}

To bound the sums appearing in sections \ref{trans} and \ref{treat}, we need to bound sums of linear sums of roughly the shape
$$
\sum\limits_{\mathcal{N}(n)\sim Z} \left|\sum\limits_{\mathcal{N}(m)\sim Y} e\left(\Im(mn\theta)\right)\right|.
$$
Considering \eqref{lin}, we  are left with bounding expressions of the form
\begin{equation} \label{Gc}
G_{\theta}(y,z):=\sum\limits_{0<N(n)\le z} \min\left\{||\Im(n \theta)||^{-1},\sqrt{y}\right\}^{1/2} \cdot 
\min\left\{||\Re(n \theta)||^{-1},\sqrt{y}\right\}^{1/2},
\end{equation}
where $y,z\ge 1$. To this end, we break the above into partial sums 
\begin{equation}
\begin{split}
& G_{\theta}(y,z,\Delta_1,\Delta_1',\Delta_2,\Delta_2')\\ := & \sum\limits_{\substack{0<N(n)\le z\\ \Delta_1< ||\Im(n \theta)|| \le \Delta_1'\\ 
\Delta_2< ||\Re(n \theta)||\le \Delta_2'}} 
\min\left\{||\Im(n \theta)||^{-1},\sqrt{y}\right\}^{1/2} \cdot \min\left\{||\Re(n \theta)||^{-1},\sqrt{y}\right\}^{1/2} 
\end{split}
\end{equation}
with $0\le \Delta_1< \Delta_1'\le 1/2$ and $0\le \Delta_2< \Delta_2'\le 1/2$ and bound them by 
\begin{equation} \label{G}
G_{\theta}(y,z,\Delta_1,\Delta_1',\Delta_2,\Delta_2')\ll  
\min\left\{\Delta_1^{-1},\sqrt{y}\right\}^{1/2} \cdot \min\left\{\Delta_2^{-1},\sqrt{y}\right\}^{1/2} \cdot  
\Sigma_{\theta}(z,\Delta_1',\Delta_2'),
\end{equation}
where
\begin{equation}\label{Sigma}
\Sigma_{\theta}(z,\Delta_1',\Delta_2') = 
\sum\limits_{\substack{0<\mathcal{N}(n)\le z\\ ||\Im(n \theta)|| \le \Delta_1'\\ ||\Re(n \theta)||
\le \Delta_2'}} 1.
\end{equation}

In the next section, we shall prove that for infinitely many Gaussian integers $q$, a bound of the form 
\begin{equation} \label{plug}
\Sigma_{\theta}(z,\Delta_1',\Delta_2')\ll  
\left(1+\frac{z}{|q|^2}\right) \cdot \left(1+\Delta_1'|q|\right)\left(1+\Delta_2'|q|\right)
\end{equation}
holds.  We shall also see that for these $q$, we have
\begin{equation} \label{also}
\Sigma_{\theta}(z,\Delta_1',\Delta_2')=0 \quad \mbox{if }
\max\{\Delta_1',\Delta_2'\}< 1/(\sqrt{8}|q|) \mbox{ and } z\le |q|^2/(8C^2) 
\end{equation}
with $C=2+\sqrt{2}$. Plugging \eqref{plug} into \eqref{G} gives
\begin{equation} \label{Gtheta}
\begin{split}
G_{\theta}(y,z,\Delta_1,\Delta_1',\Delta_2,\Delta_2')\ll &  
\left(1+\frac{z}{|q|^2}\right)\cdot \min\left\{\Delta_1^{-1},\sqrt{y}\right\}^{1/2} \cdot \min\left\{\Delta_2^{-1},\sqrt{y}\right\}^{1/2} \times\\ &  
\left(1+\Delta_1'|q|\right)\cdot \left(1+\Delta_2'|q|\right).
\end{split}
\end{equation}

Next, we write
\begin{equation*}
\begin{split}
G_{\theta}(y,z)=& G_{\theta}(y,z,0,2^{-L-1},0,2^{-L-1})+\sum\limits_{i=1}^L \sum\limits_{j=1}^L G_{\theta}(y,z,2^{-i-1},2^{-i},2^{-j-1},2^{-j})+\\ & + \sum\limits_{j=1}^L 
G_{\theta}(y,z,0,2^{-L-1},2^{-j-1},2^{-j})+\sum\limits_{i=1}^L 
G_{\theta}(y,z,2^{-i-1},2^{-i},0,2^{-L-1}),
\end{split}
\end{equation*}
where $L$ satisfies $1/(2\sqrt{y})\le 2^{-L-1}< 1/\sqrt{y}$. Using \eqref{Gtheta}, we deduce that
\begin{equation} \label{firstcase}
\begin{split}
G_{\theta}(y,z)\ll& \left(1+\frac{z}{|q|^2}\right)\left(y^{1/2}+|q|^2\right) \cdot (\log 2y)^2.
\end{split}
\end{equation}
If $z\le |q|^2/(8C^2)$, then using \eqref{also}, we have 
\begin{equation*}
\begin{split}
G_{\theta}(y,z)=& \sum\limits_{i=1}^L \sum\limits_{j=1}^L G_{\theta}(y,z,2^{-i-1},2^{-i},2^{-j-1},2^{-j})+\\ & + \sum\limits_{j=1}^L 
G_{\theta}(y,z,0,2^{-L-1},2^{-j-1},2^{-j})+\sum\limits_{i=1}^L 
G_{\theta}(y,z,2^{-i-1},2^{-i},0,2^{-L-1}),
\end{split}
\end{equation*}
where $1/(2\sqrt{8}|q|)\le 2^{-L-1}< 1/(\sqrt{8}|q|)$. In this case, using \eqref{Gtheta}, we deduce that
\begin{equation} \label{secondcase}
\begin{split}
G_{\theta}(y,z)\ll \left(|q|y^{1/4}+|q|^{2}\right)\cdot \log^2(2|q|).
\end{split}
\end{equation}

\section{Counting}
In this section, we prove \eqref{plug} and \eqref{also}. To bound the quantity $G_{\theta}(y,z)$, 
we need information about the spacing of the points $n\theta$ modulo 1, where 
$n\in \mathbb{Z}[i]$.   
We begin by using the Hurwitz continued fraction development of $\theta$ in $\mathbb{Z}[i]$ (see \cite{Hur}) to approximate $\theta$ in the form 
$$
\theta=\frac{a}{q}+\gamma,
$$
where $a,q\in \mathbb{Z}[i]$, $(a,q)=1$ and 
$$
|\gamma| \le C|q|^{-2}
$$
with $C=2+\sqrt{2}$. 
As in the classical case, this continued fraction development yields a sequence of infinitely many $q\in \mathbb{Z}[i]$ satisfying the above. 
Now it follows that 
\begin{equation} \label{dist}
\begin{split}
& \left|\left|n_1\theta-n_2\theta\right|\right|=\left|\left| \frac{(n_1-n_2)a}{q} + (n_1-n_2)\gamma\right|\right| \ge
\left|\left| \frac{(n_1-n_2)a}{q} \right|\right| - |n_1-n_2|\cdot |\gamma|\\
\ge & \frac{1}{\sqrt{2}|q|}- C\cdot \frac{|n_1-n_2|}{|q|^2} 
\end{split}
\end{equation}
if $n_1,n_2\in \mathbb{Z}[i]$ such that $n_1\not\equiv n_2\bmod{q}$. We cover the set 
$$
\mathcal{Z}:=\{n\in \mathbb{Z}[i]\ :\ 0<\mathcal{N}(n)\le z\}
$$
by $O\left(1+z/|q|^2\right)$ disjoint rectangles 
$$
\mathcal{R}=\{s\in \mathbb{C}\ :\ a_1< \Re(s)\le b_1,\ a_2< \Im(s)\le b_2\},
$$
where $|b_i-a_i|\le |q|/(4C)$, so that 
$$
\mathcal{Z}\subset \bigcup\limits_{\mathcal{R}} \mathcal{R}.
$$ 
Note that if $n_1,n_2\in \mathbb{Z}[i]\cap \mathcal{R}$, then $|n_1-n_2|\le |q|/(2\sqrt{2}C)$ and
hence, by \eqref{dist}, if $n_1,n_2\in \mathbb{Z}[i]\cap \mathcal{R}$ and $n_1\not=n_2$, then 
\begin{equation} \label{dist2}
\begin{split}
\left|\left|n_1\theta-n_2\theta\right|\right|\ge \frac{1}{2\sqrt{2}|q|}. 
\end{split}
\end{equation}

Now, 
\begin{equation*}
\Sigma_{\theta}(z,\Delta_1',\Delta_2') \le \sum\limits_{\mathcal{R}} \Sigma_{\theta}(\mathcal{R},\Delta_1',\Delta_2'),
\end{equation*}
where 
\begin{equation} \label{of}
\begin{split}
& \Sigma_{\theta}(\mathcal{R},\Delta_1',\Delta_2') := 
\sum\limits_{\substack{n\in \mathbb{Z}[i]\cap \mathcal{R}\\ ||\Im(n \theta)||\le \Delta_1'\\ ||\Re(n \theta)|| \le \Delta_2'}} 1 \\
= & \sum\limits_{\substack{n\in \mathbb{Z}[i]\cap\mathcal{R}\\ \{\Im(n \theta)\}\le \Delta_1'\\ \{\Re(n \theta)\} \le \Delta_2'}} 1 +
\sum\limits_{\substack{n\in \mathbb{Z}[i]\cap\mathcal{R}\\ \{\Im(n \theta)\}\ge 1-\Delta_1'\\ \{\Re(n \theta)\} \le \Delta_2'}} 1 +
\sum\limits_{\substack{n\in \mathbb{Z}[i]\cap\mathcal{R}\\ \{\Im(n \theta)\}\le \Delta_1'\\ \{\Re(n \theta)\} \ge 1-\Delta_2'}} 1 +
\sum\limits_{\substack{n\in \mathbb{Z}[i]\cap\mathcal{R}\\ \{\Im(n \theta)\}\ge 1-\Delta_1'\\ \{\Re(n \theta)\} \ge 1-\Delta_2'}} 1.
\end{split}
\end{equation}
If $\{\Im(n_i \theta)\}\le \Delta_1'\le 1/2$ and $\{\Re(n_i \theta)\} \le \Delta_2' \le 1/2$ for $i=1,2$, then
$$
\left|(\{\Re(n_1\theta)\}+i\{\Im(n_1\theta)\})-(\{\Re(n_2\theta)\}+i\{\Im(n_2\theta)\})\right|\ge
\left|\left|n_1\theta-n_2\theta\right|\right|,
$$
and hence, by \eqref{dist2}, if $n_1,n_2\in \mathbb{Z}[i]\cap \mathcal{R}$ and $n_1\not=n_2$, then
$$
\left|(\{\Re(n_1\theta)\}+i\{\Im(n_1\theta)\})-(\{\Re(n_2\theta)\}+i\{\Im(n_2\theta)\})\right|\ge \frac{1}{2\sqrt{2}|q|}.
$$
It follows that 
$$
\sum\limits_{\substack{n\in \mathbb{Z}[i]\cap\mathcal{R}\\ \{\Im(n \theta)\}\le \Delta_1'\\ \{\Re(n \theta)\} \le \Delta_2'}} 1\ll 
V_{1/(2\sqrt{2}|q|)}\left(\Delta_1',\Delta_2'\right),
$$
where $V_D\left(\Delta_1'\Delta_2'\right)$ is the maximal number of points of distance $\ge D$ that can be put 
into a rectangle with dimensions $\Delta_1'$ and $\Delta_2'$. The remaining three sums in the last line of \eqref{of} can be
estimated similarly. It follows that
$$
\Sigma_{\theta}(\mathcal{R},\Delta_1',\Delta_2')\ll V_{1/(2\sqrt{2}|q|)}\left(\Delta_1',\Delta_2'\right).
$$
Clearly,
$$
V_D\left(\Delta_1',\Delta_2'\right)\ll \left(1+\frac{\Delta_1'}{D}\right)\left(1+\frac{\Delta_2'}{D}\right). 
$$ 
Putting everything together, we obtain \eqref{plug}. Further, \eqref{also} holds
because $0<\mathcal{N}(n)\le |q|^2/(8C^2)$ implies
\begin{equation} 
\left|\left|n\theta\right|\right|=\left|\left| \frac{na}{q}+ n\gamma\right|\right| \ge
\left|\left| \frac{na}{q}\right|\right| - |n|\cdot |\gamma|\ge \frac{1}{\sqrt{2}|q|}- C\cdot \frac{\sqrt{|q|^2/(8C^2)}}{|q|^2}=\frac{1}{\sqrt{8}|q|}. 
\end{equation} 

\section{Final estimations of the sums $E_i$ and $F_i$} \label{finest}
Recall the conditions $H_1\ge 1$ and $H_2\ge 1/2$. Combining \eqref{F}, \eqref{lin}, \eqref{Gc} and \eqref{firstcase}, we get
\begin{equation*} 
\begin{split}
& F_3(H_1,H_2,K,K')^2 \\
\ll & (H_1H_2x)^{\varepsilon}\cdot \left(H_1^2H_2^2xK+ H_1H_2xK^{1/2}\cdot 
\left(1+\frac{\left(H_1^2+H_2^2\right)x/K}{|q|^2}\right)\left(K^{1/2}+|q|^2\right)\right),
\end{split}
\end{equation*}
where we use the facts that the number $\tau(n)$ of divisors $j$ of $n$ is $O\left(\mathcal{N}(n)^{\varepsilon}\right)$ and that the number of solutions 
$(n_1,n_2)$ with $x/(2K')<\mathcal{N}(n_1),\mathcal{N}(n_2)\le x/K$ of the equation $n/j=n_1-n_2$ is $O(x/K)$. 
Multiplying out and taking square root yields
\begin{equation} \label{F3part} 
\begin{split}
F_3(H_1,H_2,K,K')
\ll & (H_1H_2x)^{\varepsilon}\cdot \Big(H_1H_2(xK)^{1/2}+ (H_1H_2)^{1/2} \times\\
& \left((H_1+H_2)x|q|^{-1}+(H_1+H_2)x K^{-1/4}+|q| x^{1/2}K^{1/4}\right)\Big).
\end{split}
\end{equation}
Recall the definition of $F_3(H_1,H_2)$ in \eqref{F3}. From \eqref{F3part}, we conclude that
\begin{equation} \label{F3est} 
\begin{split}
F_3(H_1,H_2)
\ll & (H_1H_2x)^{\varepsilon}\cdot \Big(H_1H_2x^{(1+\alpha+\beta)/2}+ (H_1H_2)^{1/2} \times\\
& \left((H_1+H_2)x|q|^{-1}+(H_1+H_2)x^{1-\alpha/4}+|q|x^{1/2+(\alpha+\beta)/4}\right)\Big)
\end{split}
\end{equation}
by splitting the summation range of $\mathcal{N}(m)$ into $O(\log 2x)$ dyadic intervals $(K,K']$. 

We also split $E_3(H_1,H_2)$, defined in \eqref{Edef}, into $O(\log 2M)$ parts
\begin{equation*} 
E_3(H_1,H_2,K,K') := \sum\limits_{\substack{j\not=0\\ |\Re(j)|\le H_1\\ |\Im(j)|\le H_2}} \sum\limits_{K<\mathcal{N}(m)\le K'}
\Big| \sum\limits_{x/(2\mathcal{N}(m))<N(n)\le x/\mathcal{N}(m)} e\left(\Im(jmn\theta)\right)\Big|
\end{equation*}
with $1/2\le K<K'\le 2K$, which, using \eqref{lin}, \eqref{Gc} and \eqref{firstcase}, we estimate by
\begin{equation}
\begin{split}
& E_3(H_1,H_2,K,K')\\ \ll & \left(x/K\right)^{1/2}\cdot
\sum\limits_{\substack{j\not=0\\ |\Re(j)|\le H_1\\ |\Im(j)|\le H_2}} \sum\limits_{K<\mathcal{N}(m)\le K'}
\min\left\{||\Re(jm\theta)||^{-1},\sqrt{x/K}\right\}^{1/2}\cdot  \min\left\{||\Im(jm\theta)||^{-1},\sqrt{x/K}\right\}^{1/2}\\
\ll & x^{\varepsilon} \cdot \left(x/K\right)^{1/2}\cdot \sum\limits_{0<\mathcal{N}(l)\le (H_1^2+H_2^2)K'}
\min\left\{||\Re(l\theta)||^{-1},\sqrt{x/K}\right\}^{1/2}\cdot \min\left\{||\Im(l\theta)||^{-1},\sqrt{x/K}\right\}^{1/2}\\
\ll & x^{\varepsilon}\cdot \left(x/K\right)^{1/2} \cdot 
\left(1+\frac{(H_1^2+H_2^2)K}{|q|^2}\right)\left(\left(x/K\right)^{1/2}+|q|^2\right)\\
\ll & x^{\varepsilon} \cdot
\left(xK^{-1}+(H_1^2+H_2^2)x|q|^{-2}+(H_1^2+H_2^2)x^{1/2}K^{1/2}+|q|^2x^{1/2}K^{-1/2}\right).
\end{split}
\end{equation}
If $(H_1^2+H_2^2)K'\le |q|^2/(8C^2)$, then using \eqref{secondcase} instead of \eqref{firstcase}, we obtain
\begin{equation}
\begin{split}
E_3(H_1,H_2,K,K') \ll (x|q|)^{\varepsilon}\left(|q|x^{3/4}K^{-3/4}+|q|^2x^{1/2}K^{-1/2}\right).
\end{split}
\end{equation}
We deduce that for all $K\ge 1/2$, 
\begin{equation}
\begin{split}
& E_3(H_1,H_2,K,K')\ll (x|q|)^{\varepsilon} \times\\ &  \left((H_1^2+H_2^2)x|q|^{-2}+(H_1^2+H_2^2)x^{1/2}K^{1/2}+|q|x^{3/4}K^{-3/4}+
|q|^2x^{1/2}K^{-1/2}\right)
\end{split}
\end{equation}
which implies
\begin{equation} \label{E3end}
\begin{split}
& E_3(H_1,H_2)\ll (x|q|)^{\varepsilon}\times\\
& \left((H_1^2+H_2^2)x|q|^{-2}+(H_1^2+H_2^2)x^{1/2}M^{1/2}+|q|x^{3/4}+|q|^2x^{1/2}\right).
\end{split}
\end{equation}

Now, from \eqref{E3rel} and \eqref{E3end}, we obtain
\begin{equation} \label{E3ende}
E_3\ll (Jx|q|)^{\varepsilon} \cdot \left(\delta^2 J^2x|q|^{-2}+\delta^2 J^2x^{1/2}M^{1/2}+\delta^2|q|x^{3/4}+\delta^2|q|^2x^{1/2}\right),
\end{equation}
where we use the inequality
$$
\min\{\delta,H_1^{-1}\}\cdot \min\{\delta,H_2^{-1}\} \le \delta^2,
$$
and from \eqref{F3rel} and \eqref{F3est}, we obtain
\begin{equation} \label{F3ende}
F_3 \ll (Jx)^{\varepsilon}\left(x^{(1+\alpha+\beta)/2}+
\delta Jx|q|^{-1}+\delta J x^{1-\alpha/4}+\delta |q| x^{1/2+(\alpha+\beta)/4}\right),
\end{equation}
where we use the inequalities
$$
\min\{\delta,H_1^{-1}\}\cdot \min\{\delta,H_2^{-1}\} \le (H_1H_2)^{-1}
$$
and
$$
\min\{\delta,H_1^{-1}\}\cdot \min\{\delta,H_2^{-1}\} \le \delta(H_1H_2)^{-1/2}
$$
(the first for the diagonal, the second for the non-diagonal contribution).

Further, from \eqref{E1E2}, \eqref{E1H}, \eqref{E1E3} and \eqref{E3end}, we infer
\begin{equation} \label{E1E2ende}
E_1, E_2\ll (Jx|q|)^{\varepsilon} \cdot \left(\delta J^2x|q|^{-2}+\delta J^2x^{1/2}M^{1/2}+\delta|q|x^{3/4}+\delta |q|^2x^{1/2}\right), 
\end{equation}
where we use the inequality
$$
\min\{\delta,H^{-1}\}\le \delta,
$$
and from \eqref{F1F2}, \eqref{F1H}, \eqref{F1F3} and \eqref{F3ende}, we infer
\begin{equation} \label{F1F2ende}
\begin{split}
F_1, F_2\ll (Jx|q|)^{\varepsilon}\cdot
\left(x^{(1+\alpha+\beta)/2}+
\delta^{1/2} Jx|q|^{-1}+\delta^{1/2} J x^{1-\alpha/4}+\delta^{1/2} |q| x^{1/2+(\alpha+\beta)/4}\right),
\end{split}
\end{equation}
where we use the inequalities
$$
\min\{\delta,H^{-1}\}\le H^{-1}\quad \mbox{and} \quad \min\{\delta,H^{-1}\}\le \delta^{1/2}H^{-1/2}.
$$

Combining \eqref{typeI}, \eqref{E3ende} and \eqref{E1E2ende}, we obtain
\begin{equation} \label{typeIest}
\begin{split}
& \sum\limits_{\substack{\mathcal{N}(m)\le M\\ mn\in A}} a_m = 4\delta^2 \sum\limits_{\substack{\mathcal{N}(m)\le M\\ mn\in B}} a_m+\\ &
O\left((Jx|q|)^{\varepsilon}\cdot\left(\delta x J^{-1}+\delta^2 J^2x|q|^{-2}+\delta^2 J^2x^{1/2}M^{1/2}+
\delta^2|q|x^{3/4}+\delta^2|q|^2x^{1/2}\right)\right), 
\end{split}
\end{equation}
and combining \eqref{typeII}, \eqref{F3ende} and \eqref{F1F2ende}, we obtain
\begin{equation} \label{typeIIest}
\begin{split}
& \sum\limits_{\substack{x^{\alpha}<\mathcal{N}(m)\le x^{\alpha+\beta}\\ mn\in A}} a_mb_n = 
4\delta^2 \sum\limits_{\substack{x^{\alpha}<\mathcal{N}(m)\le x^{\alpha+\beta}\\ mn\in B}} a_mb_n+ \\ &
O\left((Jx)^{\varepsilon}\cdot \left(\delta xJ^{-1}+x^{(1+\alpha+\beta)/2}+
\delta Jx|q|^{-1}+\delta J x^{1-\alpha/4}+\delta |q| x^{1/2+(\alpha+\beta)/4}\right)\right). 
\end{split}
\end{equation}

Now we choose $J:=[\delta^{-1}x^{3\varepsilon}]$, $x:=|q|^{12}$ (and hence $|q|:=x^{1/12}$), $M= x^{2/3}$, 
$\alpha:=1/3$ and $\beta:=1/2$ so that
\begin{equation} \label{typeIesti}
\sum\limits_{\substack{\mathcal{N}(m)\le M\\ mn\in A}} a_m = 4\delta^2 \sum\limits_{\substack{\mathcal{N}(m)\le M\\ mn\in B}} a_m+
O\left(\delta^2x^{1-\varepsilon}+x^{5/6+8\varepsilon}\right) 
\end{equation}
and 
\begin{equation} \label{typeIIesti}
\begin{split}
\sum\limits_{\substack{x^{\alpha}<\mathcal{N}(m)\le x^{\alpha+\beta}\\ mn\in A}} a_mb_n = 
4\delta^2 \sum\limits_{\substack{x^{\alpha}<\mathcal{N}(m)\le x^{\alpha+\beta}\\ mn\in B}} a_mb_n+ O\left(\delta^2x^{1-\varepsilon}+x^{11/12+8\varepsilon}\right).
\end{split}
\end{equation}

\section{Conclusion} \label{con}
Having proved \eqref{typeIesti} and \eqref{typeIIesti}, we deduce that
\eqref{t1} and \eqref{t2} hold with $Y=\delta^2x^{1-\varepsilon}$ if $\delta\ge x^{-1/24+5\varepsilon}$. 
Now using Theorem \ref{Harsie}, \eqref{primenumber}, \eqref{Sxd} and \eqref{Sx}, it follows that
$$
\sum\limits_{\substack{x/2<\mathcal{N}(p)\le x\\ ||n\theta ||< \delta}} 1 =
4\delta^2 (1+o(1)) \cdot \sum\limits_{\substack{x/2<\mathcal{N}(p)\le x}} 1,
$$
provided that $x=|q|^{12}$, where $a/q$ is a Hurwitz continued fraction approximant of $\theta$ and
$\delta\ge x^{-1/24+\varepsilon}$ for any fixed $\varepsilon>0$. So by taking $N_k=|q_k|^{12}$, where $q_k$ is the $k$-th Hurwitz continued fraction denominator for $\theta$, 
we have the following result.

\begin{theorem} \label{approx}
Let $\theta$ be a complex number such that $\theta\not\in \mathbb{Q}(i)$ and $\varepsilon>0$ be an arbitrary 
constant. Then there exists an infinite increasing sequence of natural numbers $(N_k)_{k\in \mathbb{N}}$ such that 
\begin{equation*}
\sum\limits_{\substack{N_k/2<\mathcal{N}(p)\le N_k\\ ||p\theta ||\le\delta_k}} 1 \sim
4\delta_k^2 \cdot \sum\limits_{\substack{N_k/2<\mathcal{N}(p)\le N_k}} 1 \quad \mbox{as } k\rightarrow \infty
\end{equation*}
if $N_k^{-1/24+\varepsilon}\le \delta_k\le 1/2$. 
\end{theorem}

From this, we deduce the following.

\begin{corollary} \label{coro}
Let $\theta$ be a complex number such that $\theta\not\in \mathbb{Q}(i)$ and $\varepsilon>0$ be an arbitrary 
constant. Then there exist infinitely many Gaussian primes such that
\begin{equation*} 
||p\theta|| \le \mathcal{N}(p)^{-1/24+\varepsilon}.
\end{equation*}
\end{corollary}

\section{Notes}
(I) The bound \eqref{lin} for linear exponential sums over $\mathbb{Z}[i]$ was obtained in a very simple way by reduction to one-dimensional 
linear exponential sums. Certainly, refinements are possible under certain conditions, and this may be useful for other applications. However,
it seems that improvements of \eqref{lin} and the subsequent bounds for averages of linear exponential sums don't help in this context
because the terms that dominate here cannot be removed (for example, the term $\frac{z}{|q|^2}\cdot y^{1/2}$ in \eqref{firstcase}). 
So improvements of \eqref{lin}
will most likely not lead to progress with regard to the problem considered in this paper.\\ 

(II) It is not difficult to work out versions of our main results in which the Gaussian primes $p$ under consideration are restricted to 
fixed sectors (i.e., the argument of $p$ is restricted to some interval). All steps in the proof 
remain the same, but additional summation conditions
need to be included, which make the occuring terms even more complicated. 
We have not done this in the present note in order to keep things simple. In the author's subsequent preprint \cite{Bai}, such a more general result
is provided and the necessary alterations in its proof are indicated. \\ 

(III) It is possible to improve the exponent 1/24 in Corollary \ref{coro} using lower bound sieves.
To obtain further imrovements and improve this exponent in Theorem \ref{approx} as well, different
techniques (like bounds for Koosterman-type sums) will be required. This may be an interesting
line of future research.\\

(IV) Another interesting line could be to investigate Diophantine approximation problems of
this type for general number fields.

\end{document}